\documentclass{article}
\usepackage{amsmath,mathrsfs,amssymb,amsthm,amscd}
\usepackage{color,epsfig,latexsym,graphicx,eucal,layout,fancyhdr}
\usepackage[]{fontenc}
\usepackage[all]{xy}
\usepackage{hyperref}
\usepackage{graphics}
\usepackage{a4wide}
\usepackage{palatino}
\usepackage{setspace}

\newcounter{EQNR}
\renewcommand{\contentsname}{Contents\\
{\footnotesize\normalfont(A table of contents should normally not
be included)}}

 \theoremstyle{plain}
 \newtheorem{thm}{Theorem}[section]
 
 \numberwithin{equation}{section} 
 \theoremstyle{plain}
 \theoremstyle{plain}
 \theoremstyle{definition}
 \newtheorem{defn}[thm]{Definition}
 
 \theoremstyle{plain}
 \newtheorem{prop}[thm]{Proposition}
 \newtheorem{rem}[thm]{Remark}
 \newtheorem{lem}[thm]{Lemma}
 
 \newtheorem*{cor*}{Corollary}
 \newtheorem*{conj*}{Conjecture}
 \newtheorem*{thm*}{Theorem}

\newcommand{\bl}{\begin{lem}}
\newcommand{\el}{\end{lem}}

\newcommand{\bml}{\begin{multline}}
\newcommand{\eml}{\end{multline}}

\newcommand{\beq}{\begin{equation}}
\newcommand{\eeq}{\end{equation}}
\newcommand{\bp}{\begin{prop}}
\newcommand{\ep}{\end{prop}}
\newcommand{\bd}{\begin{defn}}
\newcommand{\ed}{\end{defn}}
\newcommand{\pf}{\begin{proof}}
\newcommand{\epf}{\end{proof}}

\newcommand{\field}[1]{\ensuremath{\mathbb{#1}}}
\newcommand{\F}{\mathcal{F}}
\newcommand{\CC}{\field{C}}

\newcommand{\NN}{\field{N}}

\newcommand{\HH}{\field{H}}

\newcommand{\RR}{\field{R}}

\let\Re\relax
\DeclareMathOperator{\Re}{Re}

\DeclareMathOperator{\vol}{vol}
\DeclareMathOperator{\R}{Re}
\DeclareMathOperator{\I}{Im}

\DeclareMathOperator{\DET}{det}
\DeclareMathOperator{\tr}{tr}

\DeclareMathOperator{\hs}{\mathcal{H}(\Gamma)}

\newcommand{\csp}{\textbf{c}}

\newcommand{\M}{\mathfrak{m}}
\newcommand{\mm}{\mathfrak{a}}
\newcommand{\Z}{\mathcal{Z}}
\newcommand{\G}{\mathcal{G}}

\begin{document}

\title{The determinant of the Lax-Phillips scattering operator}
\author{Joshua S. Friedman\footnote{The views expressed in this article are the author's own and not those of the U.S. Merchant Marine Academy,
the Maritime Administration, the Department of Transportation, or the United States government.}, Jay Jorgenson\footnote{Research supported
by NSF and PSC-CUNY grants.} and Lejla Smajlovi\'{c}}
\date{}

\maketitle

\begin{abstract}\noindent
Let $M$ denote a finite volume, non-compact Riemann surface without elliptic points, and let $B$ denote
the Lax-Phillips scattering operator.  Using the superzeta function approach due to Voros, we define
a Hurwitz-type zeta function $\zeta^{\pm}_{B}(s,z)$ constructed from the resonances associated to $zI -[ (1/2)I \pm B]$.
We prove the meromorphic continuation in $s$ of $\zeta^{\pm}_{B}(s,z)$ and, using the special value at $s=0$,
define a determinant of the operators $zI -[ (1/2)I \pm B]$.  We obtain expressions for Selberg's zeta function and
the determinant of the scattering matrix in terms of the operator determinants.
\end{abstract}

\section{Introduction}

\subsection{Determinant of the Laplacian and analytic torsion}
To begin, let $M$ denote a compact, connected Riemannian manifold of real dimension $n$ with Laplace operator $\Delta_{M}$.  Following the
seminal article \cite{MP49}, one defines the determinant of the Laplacian, which we denote by
$\DET^{*}\Delta_M$, as follows.  Let $K_{M}(t;x,y)$ be the heat kernel associated to $\Delta_M$.
Since $M$ is compact, the heat kernel is of trace class, so we can consider the trace of the
heat kernel which is given by
$$
\textrm{Tr}K_{M}(t) := \int\limits_{M}K_{M}(t;x,x)d\mu_{M}(x),
$$
where $d\mu_{M}(x)$ is the volume form on $M$.  As shown in \cite{MP49}, the parametrix construction of
the heat kernel implies that its trace $\textrm{Tr}K_{M}(t)$ admits a certain asymptotic behavior as $t$ approaches
zero and infinity, thus allowing one to define and study various integral transforms of the heat kernel.
In particular, for $s \in \mathbb{C}$ with real part $\textrm{Re}(s)$
sufficiently large, the spectral zeta function $\zeta_{M}(s)$ is defined from the Mellin transform
of the trace of the heat kernel.  Specifically, one sets
$$
\zeta_{M}(s) = \frac{1}{\Gamma(s)}\int\limits_{0}^{\infty}\left(\textrm{Tr}K_{M}(t)-1\right)t^{s}\frac{dt}{t}
$$
where $\Gamma(s)$ is the classical Gamma function.  The asymptotic expansion of $\textrm{Tr}K_{M}(t)$
as $t$ approaches zero allows one to prove the meromorphic continuation of $\zeta_{M}(s)$ to all $s \in \mathbb{C}$
which is holomorphic at $s=0$.  Subsequently, the determinant of the Laplacian is defined by

\begin{equation}\label{zeta_det}
\DET^{*}\Delta_M := \exp\left(-\zeta'_{M}(0)\right).
\end{equation}

There are several generalizations of the above considerations.  For example, let $E$ be a vector bundle on $M$, metrized
so that one can define the action of a Laplacian $\Delta_{E,k}$ which acts on $k$-forms that take values in $E$.  Analogous
to the above discussion, one can use properties of an associated heat kernel and obtain a definition of the
determinant of the Laplacian $\det^{\ast}\Delta_{E,k}$.  Going further, by following \cite{RaySing71} and \cite{RaySing73}, one
can consider linear combinations of determinants yielding, for example, the analytic torsion $\tau(M,E)$ of $E$ on $M$ which
is given by

\begin{equation}\label{zeta_torsion}
\tau(M,E) :=\frac{1}{2}\sum\limits_{k=1}^{n}(-1)^{k}k\DET^{*}\Delta_{E,k}.
\end{equation}

At this time, one understands (\ref{zeta_torsion}) to be a spectral invariant associated to the de Rham cohomology of $E$ on $M$.
If instead one considers compact, connected complex manifolds with metrized holomorphic vector bundles, one obtains a similar definition
for analytic torsion stemming from Dolbeault cohomology.

\subsection{Examples and applications}
Originally, Reidemeister-Franz torsion was an invariant defined,
under certain conditions, for any finite cell complex and orthogonal representation of its fundamental group.
As discussed in \cite{Mil66}, Reidemeister-Franz torsion is constructed from a smooth triangulation of $M$ but depends
only on the $C^{\infty}$ structure of the manifold.  Ray and Singer conjectured in \cite{RaySing71} that Reidemeister-Franz torsion
is equal to analytic torsion, and their conjecture was proved by Cheeger and M\"uller, in separate and independent work;
see \cite{Che79} and \cite{Mul78}.  It is important to note that in \cite{RaySing71} the authors showed that Reidemeister torsion can be
realized in a manner similar to (\ref{zeta_torsion}), where in that case the Laplacians are combinatorial operators.

As stated, in \cite{RaySing73} the authors extended the definition of analytic torsion, analogous to their work from \cite{RaySing71}, this
time in the setting of a compact, connected complex manifolds $M$. If $M$ is a genus one Riemann surface with flat metric, then one can
explicitly evaluate analytic torsion in terms of Dedekind's eta function; the calculation relies on explicit knowledge of the spectrum
of the Laplacian from which one can apply Kronecker's second limit formula.  If $M$ has genus $g > 1$, then it is shown in \cite{RaySing73} that
the ratio of analytic torsion for different one-dimensional unitary representations can be expressed in terms of Selberg's zeta function;
see also \cite{Jorg91} for an extension of this evaluation.

J. Fay in \cite{Fay81} proved the following fascinating connection between analytic torsion and another fundamental mathematical question.
Let $M$ be a Riemann surface of genus $g >1$, and let $\chi$ be a one-dimensional unitary representation of $\pi_{1}(M)$.  Fay proved that as a function of
$\chi$ one can extend analytic torsion to a function whose domain is $(\mathbb{C}^{\ast})^{2g}$ from which he proved that the
zero locus of the continuation determines the period matrix of $M$.  Thus, in a sense, analytic torsion is related to a type of
Torelli theorem.

Osgood, Phillips and Sarnak used properties of the determinant of the Laplacian to obtain topological results in the study of spaces
of metrics on compact Riemann surfaces, including a new proof of the uniformization theorem; see \cite{OPS89}.  In a truly fundamental
paper, Quillen \cite{Qu85} used analytic torsion to define metrics on determinant line bundles in cohomology, thus providing a means
by which Arakelov theory could be generalized from the setting of algebraic curves, as in the pioneering work of Arakelov and Faltings, to
higher dimensional considerations, as developed by Bismut, Bost, Gillet, Soul\'e, Faltings and others.  In addition, the algebraic geometric
considerations from one dimensional Arakelov theory with Quillen metrics provided a means by which physicists could study two-dimensional
quantum field theories as related to string theory; see \cite{ABMNV87}.

The evaluation of analytic torsion for elliptic curves is particularly interesting since one shows that analytic torsion can be
expressed in terms of an algebraic expression, namely the discriminant of the underlying cubic equation.   A fascinating
generalization was obtained by Yoshikawa in the setting of Enriques surfaces and certain $K3$ surfaces;
see \cite{Yo04} and \cite{Yo13}.  Other evaluations of determinants of the Laplacian in terms of holomorphic functions
can be found in \cite{KK09}, \cite{MTa06} and \cite{MTe08}.

\subsection{Non-compact hyperbolic Riemann surfaces}
If $M$ is a non-compact Riemannian manifold, then it is often the case that the corresponding heat kernel is not trace class.  Hence,
the above approach to define a determinant of the Laplacian does not get started.  This assertion is true in the case
when $M$ is a finite volume, connected, hyperbolic Riemann surfaces, which will be the setting considered in this article.
The first attempt to define a determinant of the Laplacian for non-compact, finite volume, hyperbolic
Riemann surfaces is due to I. Efrat in \cite{Ef88-91}.  Efrat's
approach began with the Selberg trace formula, which in the form Efrat employed does not connect directly with a differential
operator.  In \cite{JLund95} the authors defined a regularized difference of traces of heat
kernels, which did yield results analogous to theorems proved in the setting of compact hyperbolic Riemann surfaces.  In
\cite{Mull98}, W.~M\"uller generalized the idea of a regularized difference of heat traces to other settings.  Following
this approach, J.~Friedman in \cite{Frid07}  defined a regularized determinant of the Laplacian for any
finite-volume three-dimensional hyperbolic orbifolds with finite-dimensional unitary representations, which he then
related to special values of the Selberg zeta-function.

The concept of
a regularized quotient of determinants of Laplacians has found important applications.  For example, the dissertation of T.~Hahn
\cite{Hahn09} studied Arakelov theory on non-compact finite volume Riemann surfaces using the regularized difference
of heat trace approach due to Jorgenson-Lundelius and M\"uller; see also \cite{Freix09}.
In the seminal paper \cite{EKZ12} the authors used the regularized
difference of determinants together with the metric degeneration concept from \cite{JLund96} in their
evaluation of the sum of Lyapunov exponents of the Kontsevich-Zorich cocycle with respect to $\textrm{SL}(2,{\mathbb R})$ invariant measures.

\subsection{Our results}
It remains an open, and potentially very important, question to define determinants of Laplacians, or related spectral operators, on
non-compact Riemannian manifolds.

In the present article we consider a general finite volume hyperbolic Riemann surface without fixed points.  Scattering theory,
stemming from work due to Lax and Phillips (see \cite{Lax-Phill76} and \cite{LaxPhill80}), provides us with the definition of
a scattering operator, which we denote by $B$.  The scattering operator is defined using certain Hilbert space extensions of
the so-called Ingoing and Outgoing spaces;
see section 3 below.  Lax and Phillips have shown that $B$ has a discrete spectrum; unfortunately, one cannot define a type of
heat trace associated to the spectrum from which one can use a heat kernel type approach to defining the determinant of $B$.
Instead, we follow the superzeta function technique of regularization due to A. Voros in order
to define and study the zeta functions $\zeta^{\pm}_{B}(s,z)$ constructed from the resonances associated to $zI -[ (1/2)I \pm B]$.
We prove the meromorphic continuation in $s$ of $\zeta^{\pm}_{B}(s,z)$ and, using the special value at $s=0$,
define a determinant of the operators $zI -[ (1/2)I \pm B]$.

Our main results are as follows.  First, we obtain expressions for the Selberg zeta function and the determinant of the scattering matrix
in terms of the special values of $\zeta^{\pm}_{B}(s,z)$ at $s=0$; see Theorem \ref{thmRegZ1}.  Furthermore, we express the special value
$Z'(1)$ of the Selberg zeta function in terms of the determinant of the operator $-B+(1/2)I$; see Theorem \ref{z1_as_det}.

Regarding Theorem \ref{z1_as_det}, it is important to note the structure of the constants which relate the regularized determinant of
$-B+(1/2)I$ and $Z'(1)$.  Specifically, we now understand the nature of the corresponding constant from \cite{Sarnak} in terms of
the $R$-class of Arakelov theory.  As it turns out, the multiplicative constant which appears in Theorem \ref{z1_as_det} has a similar
structure.

\subsection{Outline of the paper}
The article is organized as follows.  In section 2 we recall various background material from the literature and establish
the notation which will be used throughout the paper.  This discussion continues in section 3 where we recall results from
Lax-Phillips scattering theory.  In section 4 we establish the meromorphic continuation of superzeta functions in very general
context.  From the general results from section 4, we prove in section 5 that the superzeta functions $\zeta^{\pm}_{B}(s,z)$ admit
meromorphic continuations, with appropriate quantifications.  Finally, in section 6, we complete the proof of the main results of
the paper, as cited above.

\section{Background material}

\subsection{Basic notation}
Let $\Gamma\subseteq\mathrm{PSL}_{2}(\mathbb{R})$ be torsion free
Fuchsian group of the first kind acting by fractional linear transformations on the upper
half-plane $\mathbb{H}:=\{z\in\mathbb{C}\,|\,z=x+iy\,,\,y>0\}$. Let $M$ be the quotient
space $\Gamma\backslash\mathbb{H}$ and $g$ the genus of $M$. Denote by $\csp$ number of inequivalent cusps of $M.$

We denote by $\mathrm{d}s^{2}_{\mathrm{hyp}}(z)$ the line element and by $\mu_
{\mathrm{hyp}}(z)$ the volume form corresponding to the hyperbolic metric on $M$
which is compatible with the complex structure of $M$ and has constant curvature
equal to $-1$. Locally on $M$, we have
\begin{align*}
\mathrm{d}s^{2}_{\mathrm{hyp}}(z)=\frac{\mathrm{d}x^{2}+\mathrm{d}y^{2}}{y^{2}}\quad
\textrm{and}\quad\mu_{\mathrm{hyp}}(z)=\frac{\mathrm{d}x\wedge\mathrm{d}y}{y^{2}}\,.
\end{align*}
We recall that the hyperbolic volume $\mathrm{vol}_{\mathrm{hyp}}
(M)$ of $M$ is given by the formula
\begin{align*}
\mathrm{vol}(M)=2\pi\bigg(2g-2+\csp\bigg).
\end{align*}

Let $\mathcal{V}^{\Gamma}$ denote the space of $\Gamma-$invariant functions $\varphi:
\mathbb{H}\longrightarrow\mathbb{C}$.
For $\varphi\in\mathcal{V}^{\Gamma}$, we set
\begin{align*}
\Vert\varphi\Vert^{2}:=\int\limits_{M}\vert\varphi(z)\vert^{2}\mu_{\mathrm{hyp}}(z),
\end{align*}
whenever it is defined. We then introduce the Hilbert space
\begin{align*}
\mathcal{H}(\Gamma):=\big\{\varphi\in\mathcal{V}^{\Gamma}\,\big\vert\,\Vert
\varphi\Vert<\infty\big\}
\end{align*}
equipped with the inner product
\begin{align*}
\langle\varphi_{1},\varphi_{2}\rangle:=\int\limits_{M}\varphi_{1}(z)\overline{\varphi_{2}(z)}\mu_{\mathrm{hyp}}(z)
\qquad(\varphi_{1},\varphi_{2}\in\mathcal{H}(\Gamma)).
\end{align*}
The Laplacian
\begin{align*}
\Delta:=-y^{2}\bigg(\frac{\partial^{2}}{\partial x^{2}}+\frac{\partial^{2}}{\partial
y^{2}}\bigg)
\end{align*}
acts on the smooth functions of $\mathcal{H}(\Gamma)$ and extends to an
essentially self-adjoint linear operator acting on a dense subspace of $\mathcal{H}(\Gamma).$

\vskip .10in
For $f(s)$ a meromorphic function, we define the \emph{null set,} $N(f) = \{ s \in \mathbb{C}~|~f(s)=0\}$ counted with multiplicity. Similarly, $P(f)$
denotes the \emph{polar set.}

\subsection{Gamma function}
Let $\Gamma(s)$ denote the gamma function. Its poles are all simple and located at each point of $-\NN,$ where $-\NN = \{ 0,-1,-2,\dots \}$.
For $|\arg{s}| \leq \pi-\delta$ and $\delta > 0$, the asymptotic expansion \cite[p. 20]{AAR99} of $\log{\Gamma(s)}$ is given by
\beq \label{gammaExpan}
\log{\Gamma(s)} = \frac{1}{2}\log{2\pi} + \left(s-\frac{1}{2}\right)\log{s} - s + \sum_{j=1}^{m-1} \frac{B_{2j}}{(2j-1)2j}\frac{1}{s^{2j-1}} + g_{m}(s).
\eeq
Here $B_i$ are the Bernoulli numbers and $g_{m}(s)$ is a holomorphic function in the right half plane $\Re(s)>0$ such that
$g_{m}^{(j)}(s) = O(s^{-2m+1-j})$ as $\Re(s)\to \infty$ for all integers $j\geq 0$, and where the implied constant depends on $j$ and $m$.

\subsection{Barnes double gamma function} \label{secBarnes}
The Barnes double gamma function is an entire order two function defined by
$$
G\left(s+1\right)=\left(2\pi\right)^{s/2}\exp\left[-\frac{1}{2}\left[\left(1+\gamma\right)s^{2}+s\right]\right]\prod_{n=1}^{\infty}\left(1+\dfrac{s}{n}\right)^{n}\exp\left[-s+\frac{s^{2}}{2n}\right],
$$
where $\gamma$ is the Euler constant. Therefore, $G(s+1)$ has a zero of multiplicity $n,$ at each point $-n \in \{-1,-2,\dots \}.$

For $s \notin -\NN$, we have that (see \cite[p. 114]{Fisher87})
\begin{equation} \label{log der barnes gamma}
\frac{G'(s+1)}{G(s+1)} = \frac{1}{2}\log(2\pi) + \frac{1}{2} - s + s\psi(s),
\end{equation}
where $\psi(s)=\frac{\Gamma'}{\Gamma}(s)$ denotes digamma function.
For $\Re(s)>0$ and as $s \rightarrow \infty,$ the asymptotic expansion of $\log G(s+1)$ is given in \cite{FL01} or\footnote{Note that \eqref{gammaExpan} is needed to reconcile these two references.  } \cite[Lemma 5.1]{AD14} by
\beq \label{asmBarnes}
\log G(s+1) = \frac{s^2}{2}\left( \log{s} - \frac{3}{2}\right) - \frac{\log{s}}{12} - s \, \zeta^{\prime}(0) + \zeta^{\prime}(-1) \>- \\
\sum_{k=1}^{n} \frac{B_{2k+2}}{4\,k\,(k+1)\,s^{2k}} +  h_{n+1}(s).
\eeq
Here, $\zeta(s)$ is the Riemann zeta-function and
$$
h_{n+1}(s)= \frac{(-1)^{n+1}}{s^{2n+2}}\int_{0}^{\infty}\frac{t}{\exp(2\pi t) -1} \, \int_{0}^{t^2}\frac{y^{n+1}}{y+s^2} \,dy \,dt.
$$
By a close inspection of the proof of \cite[Lemma 5.1]{AD14} it follows that $h_{n+1}(s)$ is holomorphic function in the right half plane $\Re(s)>0$
which satisfies the asymptotic relation $h_{n+1}^{(j)}(s) = O(s^{-2n-2-j})$ as $\Re(s)\to \infty$ for all integers $j\geq 0$,
and where the implied constant depends upon $j$ and $n$.

Set
\begin{equation} \label{def G_1}
G_{1}(s) = \left( \frac{(2\pi )^{s} (G(s+1))^{2}}{\Gamma (s)  }\right) ^{
\frac{\vol(M)}{2\pi }}
\end{equation}
It follows that $G_{1}(s)$ is an entire function of order two with zeros at points $
-n \in -\NN$ and corresponding multiplicities  $\frac{\vol(M)}{2\pi }(2n+1).$

\subsection{Hurwitz zeta function}

The Hurwitz zeta-function $\zeta_{H}(s,z)$ is defined for $\Re(s)>1$ and $z\in \CC \setminus (-\NN) $ by the absolutely convergent series

$$\zeta_{H}(s,z)=\overset{\infty }{\underset{n=0}{\sum }}\frac{1}{(z+n)^{s}}.$$

For fixed $z,$ $\zeta_{H}(s,z)$ possesses a meromorphic continuation to the whole $s-$plane with a single pole, of order 1, with residue $1$.

For fixed $z,$ one can show that
$$
\zeta_H(-n,z)= -\frac{B_{n+1}(z)}{n+1},
$$
where $n\in \NN,$ and $B_n$ denotes the $n-$th Bernoulli polynomial.

For integral values of $s$, the function $\zeta_{H}(s,z)$ is related to derivatives of the digamma function in the following way:
$$
\zeta_H(n+1,z)=\frac{(-1)^{n+1}}{n!}\psi^{(n)}(z),\quad n=1,2,...
$$

\subsection{Automorphic scattering matrix} \label{ascatmatrix}
Let $\phi(s)$ denote the determinant of the hyperbolic scattering matrix $\Phi(s),$ see \cite[\S3.5]{Venkov83}.  The function $\phi(s)$ is meromorphic of order two (\cite[Thm. 4.4.3]{Venkov83}).
It is regular for $\Re(s) > \frac{1}{2}$ except for a finite number of poles $\sigma_1, \sigma_2, \dots \sigma_{\M} \in (1/2,1]; $ each pole has multiplicity no greater than $\csp, $  the number of cusps of $M.$

We let $\rho$ denote an arbitrary pole of $\phi(s)$.  Since $\phi(s)\phi(1-s)=1,$ the set of zeros and poles are related by $N(\phi) = 1 - P(\phi),$ hence $1-\sigma_1, 1-\sigma_2, \dots 1-\sigma_{\M} $ are the zeros in $[0,1/2).$

Each pole $\sigma_{i} \in (1/2,1]$ corresponds to a $\Delta-$eigenspace,  $A_{1}(\lambda_{i}),$ with eigenvalue $\lambda_i = \sigma_{i}(1-\sigma_{i}), $ $i=1,...,\M$, in the space spanned by the incomplete theta series. For all $i=1,...,\M$ we have (\cite[Eq. 3.33 on p.299]{Hejhal83})
\beq
\text{[The multiplicity of the pole of $\phi(s)$ at $s = \sigma_{i}$] } \leq \dim A_{1}(\sigma_{i}(1-\sigma_{i})) \leq \csp. \label{eqBndPole}
\eeq

For $\Re(s)>1,$ $\phi(s) $ can be written as an absolutely convergent generalized Dirichlet series and Gamma functions;
namely, we have that
\begin{equation} \label{phiDirich}
\phi (s)=\pi ^{\frac{\csp}{2}}\left( \frac{\Gamma \left( s-\frac{1}{2}
\right) }{\Gamma \left( s\right) }\right) ^{\csp}\overset{\infty }{\underset
{n=1}{\sum }}\frac{d(n) }{g_{n}^{2s}}
\end{equation}%
where $0< g_{1} < g_{2}< ...$ and $d(n) \in \mathbb{R}$ with $d(1)\neq 0$.

We will rewrite \eqref{phiDirich} in a slightly different form. Let $c_{1}=-2\log{g_{1}} \neq 0,$  $c_{2}=\log d(1),$  and let $u_{n}=g_{n}/g_{1}>1$.
Then for $\text{Re}(s) > 1$ we can write $\phi( s) =L(s)H(s)$ where
\begin{equation} \label{eqPhiA}
L(s) =\pi^{\frac{\csp}{2}}\left( \frac{\Gamma \left( s-
\frac{1}{2}\right) }{\Gamma \left( s\right) }\right) ^{\csp} e^{c_{1}s+c_{2}}
\end{equation}
and
\begin{equation} \label{eqPhiB}
H(s) =1+\overset{\infty }{\underset{n=2}{\sum }}\frac{
a\left( n\right) }{u_{n}^{2s}},
\end{equation}
where  $a(n) \in \mathbb{R}$ and the series \eqref{eqPhiB} converges absolutely for $\Re(s)>1$.
From the generalized Dirichlet series representation \eqref{eqPhiB} of $H(s)$ it follows that
\beq \label{asmPhi}
\frac{d^{k}}{ds^{k}}\log{H(s)} = O(\beta_k^{-\Re(s)}) \quad \textrm{\rm when} \quad \Re(s) \to +\infty,
\eeq
for some $\beta_k > 1$  where the implied constant depends on $k \in \NN.$

\subsection{Selberg zeta-function} \label{szeta}
The Selberg zeta function associated to the quotient space $M=\Gamma\backslash\mathbb{H}$ is defined for $\Re(s)>1$ by
the absolutely convergent Euler product
\begin{equation*}
Z(s)=\prod\limits_{\left\{ P_0\right\} \in P(\Gamma
)}\prod_{n=0}^{\infty }\left( 1-N(P_0)^{-(s+n)}\right) \text{,}
\end{equation*}%
where $P(\Gamma )$ denotes the set of all primitive hyperbolic conjugacy classes in $\Gamma,$ and $N(P_0)$ denotes the norm of $P_0 \in \Gamma.$ From the product representation given above, we obtain for $\Re(s)>1$
\beq \label{log z(s)}
\log{Z(s)} = \sum_{\left\{ P_0\right\} \in P(\Gamma
)} \sum_{n=0}^\infty \left(-\sum_{l=1}^\infty \frac{N(P_0)^{-(s+n)l}}{l}  \right) = -
\sum_{P\in H(\Gamma )}\frac{\Lambda (P)}{N(P)^{s}\log N(P)},
\eeq
where $H(\Gamma )$ denotes the set of all hyperbolic conjugacy classes in $\Gamma,$ and $\Lambda (P)=\frac{\log N(P_{0})}{1-N(P)^{-1}}$, for the
(unique) primitive element $P_{0}$ conjugate to $P$.

Let $P_{00}$ be the primitive hyperbolic conjugacy class in all of $P(\Gamma )$ with the smallest norm. Setting $\alpha = N(P_{00})^{\tfrac{1}{2}}$,
we see that for $\Re(s) > 2$ and $k \in \NN$ the asymptotic
\beq \label{eqSelZetaBound}
\frac{d^{k}}{ds^{k}}\log{Z(s)} = O(\alpha^{-\Re(s)}) \quad \textrm{\rm when} \quad \Re(s) \to +\infty.
\eeq
Here the implied constant depends on $k \in \NN.$

We now state the divisor of the $Z(s)$  (see \cite[p. 49]{Venkov90}  \cite[p. 499]{Hejhal83}):

\begin{enumerate}
\item Zeros at the points  $s_j$ on the line $\Re(s)=\tfrac{1}{2}$ symmetric relative to the real axis and in  $(1/2,1].$  Each zero $s_j$ has multiplicity $m(s_j) = m(\lambda_j)$ where $s_j(1-s_j) = \lambda_j$ is an eigenvalue in the discrete spectrum of $\Delta$;  \label{szeta1}

\item Zeros at the points $s_{j} = 1-\sigma_{j} \in [0,1/2)$ (see \S~\ref{ascatmatrix}). Here, by \eqref{eqBndPole}, the multiplicity $m(s_{j})$ is
$[\text{multiplicity of the eigenvalue $\lambda_{j} = \sigma_j(1-\sigma_j)$ }] - [\text{order of the pole of $\phi(s)$ at $s=\sigma_{j}$ } ] \geq 0$;  \label{szeta2}

\item If $\lambda = \tfrac{1}{4}$ is an eigenvalue
of $\Delta $ of multiplicity $d_{1/4}$, then $s=\tfrac{1}{2}$ is a zero (or a pole, depending on the sign of the following) of $Z(s)$
of multiplicity $$2d_{1/4}- \tfrac{1}{2}\left( \csp- \tr \Phi (\tfrac{1}{2})\right); $$

\item Zeros at each $s = \rho,$ where $\rho$ is a pole of $\phi(s)$ with $\Re(\rho) < \tfrac{1}{2}$; \label{pzeta}

\item \emph{Trivial} Zeros at points $s=-n \in -\NN$,  with multiplicities $\tfrac{
\vol(M)}{2\pi }(2n+1)$;

\item Poles at $s=-n-\tfrac{1}{2}, $ where $n=0,1,2,\dots,$  each with
multiplicity $\csp$.
\end{enumerate}

\subsection{Selberg zeta function of higher order}

For $\R(s)>1$ and $r \in \mathbb{N}$, following \cite{KurWakYam}, Section 4.2. we define the \emph{Selberg zeta function of order} $r$, or the \emph{poly-Selberg zeta function of degree} $r,$ by the relation
\begin{equation*}
Z^{(r)}(s)=\exp \left( -\sum_{P\in H(\Gamma )}\frac{\Lambda (P)}{%
N(P)^{s}\left( \log N(P)\right) ^{r}}\right).
\end{equation*}%
This definition is consistent with the case $r=1$ (see Equation~\ref{log z(s)}), namely $Z^{(1)}(s)=Z(s)$.

Following \cite{KurWakYam}, Section 4.2. it is easy to show that
\begin{equation*}
Z^{(r)}(s)=\prod\limits_{\left\{ P_{0}\right\} \in P(\Gamma
)}\prod_{n=0}^{\infty }H_{r}\left( N(P_{0})^{-(s+n)}\right) ^{\left( \log
N(P_{0})\right) ^{-(r-1)}},
\end{equation*}%
for $\R(s)>1$, where $H_{r}(z)=\exp (-\text{Li}_{r}(z)),$ and $$\text{Li}_{r}(z)=\sum_{k=1}^{\infty }%
\frac{z^{k}}{k^{r}}  \quad (\left\vert z\right\vert <1 )$$ is the polylogarithm
of a degree $r.$

The meromorphic continuation of $Z^{(r)}(s)$ follows inductively for $r\in\NN$ from the
differential ladder relation
$$
\frac{d^{r-1}}{dz^{r-1}}\log Z^{(r)}(s)=(-1)^{r-1}\log Z(s).
$$
See \cite[Proposition 4.9]{KurWakYam} for more details. Note that \cite{KurWakYam} deals with compact Riemann surfaces, so one must modify the region $\Omega _{\Gamma }$ defined in \cite[Proposition
4.9]{KurWakYam} by excluding the vertical lines passing through poles $\rho$ of the hyperbolic scattering determinant $\phi$; the other details are identical.

\subsection{Complete zeta functions}  \label{complete zetas}

In this subsection we define two zeta functions $Z_+(s)$ and $Z_-(s)$ associated with $Z(s)$ which are both entire functions of order two.

Set $$ Z_+(s)=\frac{Z(s)}{G_1(s)(\Gamma(s-1/2))^{\csp}},  $$
where $G_1(s)$ is defined by \eqref{def G_1}.
Note that we have canceled out the trivial zeros and poles of $Z(s),$ hence the set $N(Z_+)$ consists of the following:

\begin{itemize}
\item At $s = \tfrac{1}{2},$ the multiplicity of the zero is $\mm,$ where
$$\mm = 2d_{1/4}+\csp-\tfrac{1}{2}\left( \csp-\tr \Phi (\tfrac{1}{2})\right) = 2d_{1/4}+\tfrac{1}{2} \left( \csp +\tr \Phi (\tfrac{1}{2})\right) \geq 0;$$

\item Zeros at the points  $s_j$ on the line $\Re(s)=\tfrac{1}{2}$ symmetric relative to the real axis and in  $(1/2,1].$  Each zero $s_j$ has multiplicity $m(s_j) = m(\lambda_j)$ where $s_j(1-s_j) = \lambda_j$ is an eigenvalue in the discrete spectrum of $\Delta$;

\item Zeros at the points $s_{j} = 1-\sigma_{j} \in [0,1/2)$  (see \S~\ref{ascatmatrix}). Here, by \eqref{eqBndPole}, the multiplicity $m(s_{j})$ is
$$[\text{multiplicity of the eigenvalue $\lambda_{j} = \sigma_j(1-\sigma_j)$ }] - [\text{order of the pole of $\phi(s)$ at $s=\sigma_{j}$ } ] \geq 0;$$

\item Zeros at each $s = \rho,$ where $\rho$ is a pole of $\phi(s)$ with $\Re(\rho) < \tfrac{1}{2}$.

\end{itemize}

Set $$Z_-(s)= Z_+(s)\phi(s).$$

Then it follows that $N(Z_-) = 1-N(Z_+).$ That is, $s$ is a zero of $Z_+$ iff $1-s$ is a zero (of the same multiplicity) of $Z_-.$

\section{Lax-Phillips scattering operator on $M$}




Following \cite{Lax-Phill76} and \cite{PS92} we will introduce the scattering operator $B$ on $M$ and identify its spectrum. Let  $u = u(z,t)$ be a smooth function on $\HH \times \RR.$ Consider the hyperbolic wave equation for $
-\Delta, $
\begin{equation*}
u_{tt}=Lu=-\Delta u-\frac{u}{4},
\end{equation*}
with initial values $f=\left\{ f_{1},f_{2}\right\} \in \hs \times \hs  $, where
\begin{equation*}
u(z,0)=f_{1}(z)\text{ \ \ \ and \ \ \ }u_{t}(z,0)=f_{2}(z).
\end{equation*}

Recall that $\langle \cdot,\cdot \rangle$ is the inner product on $\hs.$ The energy form (norm) for the wave equation is
$$E(u) = \left< u,Lu \right> + \left< \partial_{t} u, \partial_{t} u \right>. $$
The energy form is independent of $t,$ so in terms of initial values, an integration by parts yields
\begin{equation*}
E(f)=\underset{\F}{\int }\left( y^{2}\left\vert \partial
f_{1}\right\vert ^{2}-\frac{\left\vert f_{1}\right\vert ^{2}}{4}+\left\vert
f_{2}\right\vert ^{2}\right) \frac{dxdy}{y^{2}},
\end{equation*}
where $\F$ denotes the Ford fundamental domain of $\Gamma$.

In general, the quadratic form $E$ is not positive definite. To overcome this difficulty we follow \cite{PS92} and modify  $E$ in the following manner: Choose a partition of unity $\{ \psi_{j}~|~j=0,\dots, \csp \}$ with $\psi_{0}$ of compact support and $\psi_{j}=1$ in the $j$th cusp (transformed to $\infty$) for $y > a,$ where $a$ is fixed
and sufficiently large. Set
$$ E_{j}(f)=\underset{\F}{\int }\psi_{j} \left( y^{2}\left\vert \partial
f_{1}\right\vert ^{2}-\frac{\left\vert f_{1}\right\vert ^{2}}{4}+\left\vert
f_{2}\right\vert ^{2}\right) \frac{dxdy}{y^{2}},$$ so that
$E = \sum_{j} E_{j}.$ There exists a constant $k_{1}$ and a compact subset $K \subset \F$ so that $$G(f) := E(f) + k_1 \int_{K} |f_{1}|^{2} \frac{dxdy}{y^{2}}$$ is positive definite\footnote{ \cite[p. 4]{PS92} and \cite[p. 265]{LaxPhill80} differ in the $y^{-2}$ term.
 }.

Define the Hilbert space $\hs_{G}$ as the completion  with respect to $G$ of $C^{\infty}$ data $f = \{f_{1},f_{2}\} \in C_{0}^{\infty
}(\F)\times C_{0}^{\infty }(\F)$ with compact support.

The wave equation may be written in the form $f_{t}=Af$
where
$$A=\left(
\begin{array}{cc}
0 & I \\
L & 0%
\end{array}%
\right), $$
defined as the closure of $A$, restricted to $C_{0}^{\infty
}(\F)\times C_{0}^{\infty }(\F)$. The operator $A$ is the infinitesimal generator a unitary group $U(t)$ with respect to the energy norm $E.$

The Incoming and Outgoing subspaces of $\hs_{G}$ are defined as follows.
\begin{itemize}
\item The Incoming subspace $\mathcal{D}_{-}$ is the closure in $\hs_{G}$ of the set of elements of the form $\{y^{1/2}\varphi(y), y^{3/2}\varphi'(y) \},$
where $\varphi$ is a smooth function of $y$ which vanishes for $y \leq a,$ and $\varphi' = \frac{d}{dy}\varphi$.
 \item The Outgoing subspace $\mathcal{D}_{+}$ is defined analogously as the closure of $\{y^{1/2}\varphi(y), -y^{3/2}\varphi'(y) \}.$
\end{itemize}
The subspaces $\mathcal{D}_{-}$ and $\mathcal{D}_{+}$ are $G$ orthogonal. Let $\mathcal{K}$ denote the orthogonal complement of $\mathcal{D}_{-}\oplus \mathcal{D}_{+}$ in $\mathcal{H}(\Gamma)_{G}$ and let $P$
denote the $G$-orthogonal (and $E$-orthogonal\footnote{Since the functions $\phi(y)$ are zero outside of the cusp sectors, the $E$ and $G$ forms agree.}) projection of
$ \mathcal{H}(\Gamma)_{G}$ onto $\mathcal{K}$ and set
\begin{equation*}
\boldsymbol{Z}(t)=PU(t)P\text{, \ \ \ for }t\geq 0.
\end{equation*}%

The operators $\boldsymbol{Z}(t)$ form a strongly continuous semigroup of operators on $\mathcal{K}$ with infinitesimal generator $B$. For every $\lambda$ in the resolvent set of $B,$ $(B-\lambda I)^{-1}$ is a compact operator \cite[Sec. 3]{LaxPhill80}. Hence, $B$ has a pure point spectrum of finite multiplicity and $(B-\lambda I)^{-1}$ is meromorphic in the entire complex plane. See also \cite[Thm. 2.7]{Lax-Phill76}.

Following \cite{PS92}, we define \textit{the singular set} $\sigma(\Gamma)$. First, we define the multiplicity function $m(r)$ as follows:
\begin{enumerate}
\item If $\I(r) \leq 0$ and $r \neq 0,$ the multiplicity $m(r)$ is the dimension of the eigenspace for $\lambda = \frac{1}{4}+r^{2}$ for $\Delta$ on $M = \Gamma \setminus \HH.$  Hence for $\I(r) \leq 0, m(r) = 0$ outside of $(-\infty,\infty) \cup -i(0,\frac{1}{2}].$
\item If $\I(r)>0,$ $m(r)$ is the multiplicity of the eigenvalue $\frac{1}{4}+r^{2}$ plus the order of the pole (or negative the order of the zero) of $\phi(s)$ at $s = \frac{1}{2}+ir.$
\item For $r=0,$ $m(r)$ is twice the multiplicity of the cusp forms (with eigenvalue $\lambda = 1/4$) plus $(\csp + \text{tr}(\Phi(1/2))/2.$
\end{enumerate}

Then, the singular set $\sigma(\Gamma)$ is defined to be the set of all $r \in \mathbb{C}$ with $m(r) > 0,$ counted with multiplicity. The singular set  $\sigma(\Gamma)$  is closely related to the spectrum $\mathrm{Spec}(B)$ of the operator $B$ by the equation $\mathrm{Spec}(B) = i\sigma(\Gamma),$ see  \cite{PS92}.

Therefore, by setting $s = \tfrac{1}{2}+ ir$ and referring to \S\ref{complete zetas}, we have
\beq
\text{Spec}\left( \frac{1}{2}I+B\right) =N (Z_+).
\eeq
and
\beq
\text{Spec}\left( \frac{1}{2}I-B\right) = N(Z_-).
\eeq

\section{Process of zeta regularization}

In the mathematical literature, there exist  mainly three different approaches to zeta regularization. In the abstract approach,
as in \cite{Illies01}, \cite{JL93b}, \cite{KimWak04}
and \cite{KurWak04} the authors start with a general sequence of complex numbers
(generalized eigenvalues) and define criteria for the zeta regularization
process. For example, in \cite{JL93b}, a theta series is introduced and, under suitable conditions at zero and infinity, a possibly regularized
zeta function is defined as the Laplace-Mellin transform of the theta series.

The second approach is based on a generalization of the Poisson summation formula or
explicit formula. Starting with the truncated heat kernel, one defines a regularized zeta function as the Mellin transform of the trace of the truncated heat kernel modulo the factor $\frac{1}{\Gamma (s)}$. Variants of the second approach can be found in \cite{RaySing73}, \cite{Sarnak} , \cite{Ef88-91}, \cite{Mull92}, \cite{Mull98}, \cite
{MulMul06}, \cite{Frid07} and many others.

The third approach, formulated by A. Voros in \cite{Voros1}, \cite{Voros2}, \cite{Voros3}, and \cite{VorosKnjiga} is based on the construction
of the so-called superzeta functions, meaning zeta functions constructed over a set of zeros of the primary zeta function. In this setting, one
starts with a sequence of zeros, rather
than the sequence of eigenvalues, of a certain meromorphic function and then
induces zeta regularization through meromorphic continuation of
 an integral representation of this function, valid in a certain strip. In this section we give a brief description of this methodology.

Let $\RR^{-} = (-\infty,0]$ be the non-positive real numbers. Let $\{y_{k}\}_{k\in \mathbb{N}}$ be the sequence of zeros
of an entire function $f$ of order two, repeated with their multiplicities. Let
$$X_f = \{z \in \CC~|~ (z-y_{k}) \notin \RR^{-}~\text{for all} ~ y_{k} \}. $$
For $z \in X_f,$ and $s \in \CC$ (where convergent) consider the series
\begin{equation}
\Z_{f}(s,z)=\sum_{k=1}^{\infty }(z-y_{k})^{-s},  \label{Zeta1}
\end{equation}
where the complex exponent is defined using the principal branch of the logarithm with $\arg z\in
\left( -\pi ,\pi \right) $ in the cut plane $\CC \setminus \RR^{-}. $

Since $f$ is of order two, $\Z_{f}(s,z)$ converges absolutely for $\Re(s) > 2.$
The series $\Z_{f}(s,z)$ is called the zeta function associated to the zeros of $f $, or the simply the \emph{superzeta} function of $f.$

If $\Z_{f}(s,z)$ has a meromorphic continuation which is regular at $s=0,$ we define the \emph{zeta regularized product} associated to $f$ as
$$ D_{f}\left( z \right) = \exp\left( {-\frac{d}{ds}\left. \Z_{f}\left( s,z\right) \right|_{s=0}  } \right).$$

Hadamard's product formula allows us to write
\beq
f(z) = \Delta_{f}(z) = e^{g(z)} z^r \prod_{k=1}^\infty \left( \left(1-\frac{z}{y_k}     \right)\exp\left[ \frac{z}{y_k} + \frac{z^2}{2{y_k}^2}   \right]    \right),
\eeq
where $g(z)$ is a polynomial of degree 2 or less, $r\geq 0$ is the order of eventual zero of $f$ at $z=0,$ and the other zeros $y_k$ are listed with multiplicity. A simple calculation shows that when $z \in X_f,$
\beq \label{eqTripDer}
\Z_{f}(3,z)=\frac{1}{2}\left(\log \Delta _{f}\left( z\right) \right)
^{\prime \prime \prime }.
\eeq

The following proposition is due to Voros (\cite{Voros1}, \cite{Voros3}, \cite{VorosKnjiga}). For completeness, we give a different proof.

\begin{prop} \label{prop: Voros cont.}
Let $f$ be an entire function of order two, and for $k\in\NN,$ let $y_k$ be the sequence of zeros of $f.$ Let $\Delta_{f}(z)$
denote the Hadamard product representation of $f.$ Assume that for $n>2$ we have the following asymptotic expansion:
\begin{equation} \label{defAE}
\log \Delta_{f}(z)= \widetilde{a}_{2}z^{2}(\log z-\frac{3}{2}%
)+b_{2}z^{2}+\widetilde{a}_{1}z\left( \log z-1\right) +b_{1}z+\widetilde{a}%
_{0}\log z+b_{0}+\sum_{k=1}^{n-1}a_{k}z^{\mu _{k}} + h_n(z),
\end{equation}
where $1>\mu _{1}>...>\mu _{n} \rightarrow -\infty $, and $h_n(z)$ is a sequence of holomorphic functions in the sector $\left\vert \arg z\right\vert <\theta <\pi, \quad (\theta >0)$ such that $h_n^{(j)}(z)=O(|z|^{\mu_n-j})$, as $\left\vert z\right\vert \rightarrow \infty $ in the above sector, for all integers $j \geq 0.$

Then, for all $z\in X_f,$ the superzeta function $\Z_{f}(s,z)$  has a meromorphic continuation to the half-plane $\Re(s)<2$ which is regular at $s=0.$

Furthermore, the zeta regularized product $D_{f}\left( z\right) $ associated to $\Z_{f}(s,z)$ is related to $\Delta_{f}(z)$ through the formula
\begin{equation}
D_{f}(z)=e^{-(b_{2}z^{2}+b_{1}z+b_{0})}\Delta_{f}(z).  \label{D(z)}
\end{equation}
\end{prop}
\begin{proof}

For any $z\in X_f$, the series
\beq
\Z_{f}(3, z+y)= \sum_{k=1}^{\infty }(z+y-y_{k})^{-3} \label{zeta2}
\eeq
converges absolutely and uniformly for  $y \in (0,\infty).$  Furthermore, application of \cite[Formula 3.194.3]{GR07}, with $\mu=3-s$, $\nu = 3$ and $\beta=(z-y_k)^{-1}$ yields, for all $y_k,$
$$
\int\limits_{0}^{\infty}\frac{y^{2-s}\,dy }{(z+y-y_k)^3 }=\frac{1}{2}(z-y_k)^{-s}\Gamma(3-s)\Gamma(s).
$$
Absolute and uniform convergence of the series \eqref{zeta2} for $\Re(s)>2$ implies that
$$
\Z_{f}(s,z)=\frac{
2}{\Gamma(3-s)\Gamma(s)}\int_{0}^{\infty }\Z_{f}(3,z+y)y^{2-s}dy,
$$
for $2< \Re(s) <3.$ From the relation
\beq \notag
\frac{1}{\Gamma (s)\Gamma (3-s)}=\frac{1}{\Gamma (s)\Gamma (1-s)(1-s)(2-s)}=%
\frac{\sin \pi s}{\pi (1-s)(2-s)},
\eeq
(which is obtained by the functional equation and the reflection formula for the gamma function) we obtain
\begin{equation}
\Z_{f}(s,z) = \frac{
2\sin \pi s}{\pi (1-s)(2-s)}\int_{0}^{\infty }\Z_{f}(3,z+y)y^{2-s}dy\text{,}
\label{Zintegralnareprezentacija}
\end{equation}%
for $2< \Re(s) <3$.

Next, we use \eqref{Zintegralnareprezentacija} together with \eqref{defAE} in order to get the meromorphic continuation of $Z_{f}(s,z)$ to the half plane Re$(s)<3$.
We start with \eqref{eqTripDer} and differentiate Equation~(\ref{defAE}) three times to get
$$
\Z_{f}(3,z+y)= \frac{\widetilde{a}_2}{(z+y)}- \frac{\widetilde{a}_1}{2(z+y)^2}+\frac{\widetilde{a}_0}{(z+y)^3}+\sum_{k=1}^{n-1}\frac{a_{k}\mu_k(\mu_k-1)(\mu_k-2)}{2(z+y)^{3-\mu _{k}}} + \frac{1}{2}h_n'''(z+y),
$$
for any $n>2$.

Since $\mu_k \searrow -\infty$, for an arbitrary $\mu<0$ there exists $k_0$ such that $\mu_k \leq \mu$ for all $k\geq k_0$, hence we may write
$$
\Z_{f}(3,z+y)y^3 = y^3\left(\frac{\widetilde{a}_2}{(z+y)}- \frac{\widetilde{a}_1}{2(z+y)^2}+\frac{\widetilde{a}_0}{(z+y)^3} + \sum_{k=1}^{k_0-1}\frac{a_{k}\mu_k(\mu_k-1)(\mu_k-2)}{2(z+y)^{3-\mu _{k}}}\right) + g_{\mu}(z+y),
$$
where $g_{\mu}(z+y)=\frac{1}{2}y^3h_{k_0}'''(z+y).$

Note that
\beq g_{\mu}(z+y) = O(y^{\mu}) \quad \text{as $y \to \infty,$} \quad \text{and} \quad g_{\mu}(z+y) = O(y^3) \text{ as $y \searrow 0$}. \label{eqDecayG}
\eeq
Application of \cite[Formula 3.194.3]{GR07} yields
\begin{multline}\label{z int for continuation}
\int_{0}^{\infty }\Z_{f}(3,z+y)y^{2-s}dy= \widetilde{a}_2z^{2-s}\Gamma(3-s)\Gamma(s-2)-\frac{\widetilde{a}_1}{2}z^{1-s}\Gamma(3-s)\Gamma(s-1)
 + \frac{\widetilde{a}_0}{2}\Gamma(3-s)\Gamma(s)\\
+ \sum_{k=1}^{k_0-1}a_{k}\mu_k(\mu_k-1)(
\mu_k-2)\frac{\Gamma(3-s)\Gamma(s-\mu_k)}{2\Gamma(3-\mu_k)}z^{\mu_k-s}
 + \int_{0}^{\infty }g_{\mu}(z+y)y^{-s-1}dy.
\end{multline}
The integral on the right hand side of \eqref{z int for continuation} is the Mellin transform of the function $g_{\mu}.$ By \eqref{eqDecayG} this integral represents a holomorphic function in $s$ for all $s$ in the half strip $\mu < \mathrm{Re}(s)<3$. The other terms on the right hand side of \eqref{z int for continuation} are meromorphic in $s$, hence the right-hand side of \eqref{z int for continuation} provides meromorphic continuation of integral $\int_{0}^{\infty }Z_{f}(3,z+y)y^{2-s}dy$ from the strip $2<\mathrm{Re}(s)<3$ to the strip $\mu < \mathrm{Re}(s)<3$. Since $\mu<0$ was chosen arbitrarily, we can let $\mu \to -\infty$ and obtain the meromorphic continuation of this integral to the half plane Re$(s)<3.$

Formula \eqref{z int for continuation}, together with \eqref{Zintegralnareprezentacija}, after multiplication with $\frac{2}{\Gamma(s)\Gamma(3-s)},$ now yields the following representation of $\Z_f(s,z)$, for an arbitrary, fixed $z\in X_f,$ valid in the half plane $\mu<\mathrm{Re}(s)<3$:
\begin{multline} \label{Zf repres}
\Z_f(s,z)= \frac{2\widetilde{a}_2}{(s-1)(s-2)}z^{2-s}-\frac{\widetilde{a}_1}{(s-1)}z^{1-s} +
\widetilde{a}_0z^{-s} - \sum_{k=1}^{k_0-1}a_{k}\frac{\Gamma(s-\mu_k)}{\Gamma(s)\Gamma(-\mu_k)}z^{\mu_k-s}\\ + \frac{1}{\Gamma(s)\Gamma(3-s)}\int_{0}^{\infty }h_{k_0}'''(z+y)y^{2-s}dy.
\end{multline}
From the decay properties of $h_{k_0}'''(z+y),$ it follows that $\Z_f(s,z)$ is holomorphic at $s=0.$  Furthermore since  $\tfrac{1}{\Gamma(s)}$ has a zero at $s=0,$  the derivative of the last term in \eqref{Zf repres} is equal to
$$
\left( \left. \frac{d}{ds}\frac{1}{\Gamma(s)}\right|_{s=0}\right) \frac{1}{\Gamma(3)}\int_{0}^{\infty }h_{k_0}'''(z+y)y^{2}dy = - \frac{1}{2}\int_{0}^{\infty }h_{k_0}'''(z+y)y^{2}dy = h_{k_0}(z),
$$
where the last equality is obtained from integration by parts two times, and  using the decay of $h_{k_0}(z+y)$ and its derivatives as $y\to+\infty$, for $\mu_{k_0}<0.$ Moreover, since
$$
\left.\frac{d}{ds}\frac{\Gamma(s-\mu_k)}{\Gamma(s)} \right|_{s=0}= \lim_{s\to 0} \frac{\Gamma(s-\mu_k)}{\Gamma(s)} \cdot \frac{\Gamma'}{\Gamma}(s)=-\Gamma(-\mu_k),
$$
elementary computations yield
$$
-\left. \frac{d}{ds} \Z_f(s,z) \right|_{s=0}=\widetilde{a}_{2}z^{2}(\log z-\frac{3}{2}%
)+\widetilde{a}_{1}z\left( \log z-1\right) +\widetilde{a}%
_{0}\log z+\sum_{k=1}^{k_0-1}a_{k}z^{\mu _{k}} + h_{k_0}(z),
$$
for $z$ in the sector $\left\vert \arg z\right\vert <\theta <\pi $, $(\theta >0)$. Finally, \eqref{D(z)} follows from the uniqueness of analytic continuation.
\end{proof}

\section{Polar structure of superzeta functions associated to $Z_+$ and $Z_-$}

Recall the definitions of $Z_+, Z_-,$ $G_1,$ and the null sets $N(Z_{\pm}).$

Set $X_{\pm} = X_{Z_{\pm}},$ and for $z \in X_{\pm},$ denote by  $\zeta_{B}^{\pm}(s,z):=\mathcal{Z}_{Z_{\pm}}(s,z)$ the superzeta functions of $Z_{\pm}.$

In this section we prove that $\zeta_{B}^{\pm}(s,z)$ has a meromorphic continuation to all $s\in\CC$, with simple poles at $s = 2 $ and $s=1,$ and we  determine the corresponding residues.

Let $\G_1(s,z)$ be the superzeta function associated to the $G_1(s)$, defined for $z \in X_{G_1}=\CC\setminus \RR^{-},$ and $\Re(s) > 2$ by
\beq \label{eqG1Hur}
\G_1(s,z) = \frac{\vol(M)}{2\pi} \sum_{n=0}^\infty \frac{(2n+1)}{(z+n)^{s}} = \frac{\vol(M)}{\pi }\left[ \zeta
_{H}(s-1,z)-(z-1/2)\zeta _{H}(s,z)\right].
\eeq

Equation~\ref{eqG1Hur} and the meromorphic continuation of  $\zeta_{H}(s,z)$ immediately yield

\begin{prop} \label{prop:cont of G1}
For for $z \in \CC\setminus \RR^{-}, $ function  $\G_{1}(s,z)$
admits a meromorphic continuation (in the $s$ variable) to $\CC$ with simple poles at $
s = 2 $ and $s=1,$ with corresponding residues $\frac{\vol(M)}{\pi }$ and $-\frac{%
\vol(M)}{2\pi }(2z-1)$, respectively.
\end{prop}

Recall the divisor of the Selberg zeta-function $Z(s)$ in \S\ref{szeta} and note that
$\{z \in \CC~|~ (z-w_{k}) \notin \RR^{-}~\text{for all} ~ w_{k} \}=X_{+},$ where $w_k$ is a zero or a pole of $Z(s)$. Analogously, the set $\{z \in \CC~|~ (z-y_{k}) \notin \RR^{-}~\text{for all} ~ y_{k} \},$ where $y_k$ is a zero or a pole of $ZH(s)$ is equal to $X_{-}$. The polar structure of the superzeta function $\zeta_B^+(s,z)$ is given as follows:

\begin{thm} \label{repZ1Thm} Fix $z\in X_+. $ The superzeta
function $\zeta_B^+(s,z)$ has meromorphic continuation to all $s \in \CC,$ and satisfies
\begin{equation} \label{Z1rep}
\zeta_B^+(s,z)=-\mathcal{G}_{1}(s,z)+ \csp\zeta_{H}(s,z-\tfrac{1}{2}) +\frac{\sin \pi s}{\pi }\int_{0}^{\infty }\left( \frac{Z^{\prime }}{Z}%
(z+y)\right) y^{-s}dy.
\end{equation}
Furthermore, the function $\zeta_B^+(s,z)$ has two simple poles at $s=1$ and $s=2$ with corresponding residues $\frac{\vol(M)}{\pi}(z-1/2) + \csp$ and $-\frac{\vol(M)}{\pi}$ respectively.
\end{thm}

\begin{proof}
For $ z\in X_+ $ and  $2 < \R(s) <3,$ we apply Proposition~\ref{prop: Voros cont.} and Equation~\ref{eqTripDer} to get
\begin{gather} \label{zeta B plus G_1}
\zeta_B^+(s,z)+\mathcal{G}_{1}(s,z)=\frac{2\sin \pi s}{\pi
(1-s)(2-s)}\int_{0}^{\infty }\left[ \zeta_B^+(3,z+y)+\mathcal{G}%
_{1}(3,z+y)\right] y^{2-s}dy  \\
=\frac{\sin \pi s}{\pi (1-s)(2-s)}\int_{0}^{\infty }\left( \log F\left(
z+y\right) \right) ^{\prime \prime \prime }y^{2-s}dy=\frac{\sin \pi s}{\pi
(1-s)(2-s)}\int_{0}^{\infty }y^{2-s}d\left( \left( \log F\left(
z+y\right) \right) ^{\prime \prime }\right), \notag
\end{gather}
where we put $F(x)= Z_+(x)G_1(x)$, hence, according to \S\ref{complete zetas}
$$\log(F(x)) = \log{Z(x)}- \csp \log\left(\Gamma(x-\tfrac{1}{2})\right),$$
and
\begin{equation*}
\left( \log F\left( z+y\right) \right) ^{\prime \prime }=-\csp\psi
^{\prime }(z+y-1/2)+\left( \frac{Z^{\prime }(z+y)}{Z(z+y)}\right)
^{\prime }.
\end{equation*}
For fixed $z\in X_+,$ it follows from (\eqref{eqSelZetaBound} and \eqref{gammaExpan}) that
\begin{equation*}
\left( \log F\left( z+y\right) \right) ^{\prime \prime }=O\left(\frac{1}{y}\right)%
\text{, \ as \ }y\rightarrow \infty
\end{equation*}%
and
\begin{equation*}
\left( \log F\left( z+y\right) \right) ^{\prime \prime }=O(1)\text{, \
as \ } y  \searrow 0.
\end{equation*}%
Therefore, for $1< \Re(s) < 2$ we may integrate by parts and obtain%
\begin{gather} \label{ZplusGdrugo}
\frac{\sin \pi s}{\pi (1-s)(2-s)}\int_{0}^{\infty }y^{2-s}d\left( \left(
\log F\left( z+y\right) \right) ^{\prime \prime }\right)= \\ -\frac{\sin
\pi s}{\pi (1-s)}\int_{0}^{\infty }\left( \frac{Z^{\prime }(z+y)}{%
Z(z+y)}\right) ^{\prime }y^{1-s}dy
+\csp \frac{\sin \pi s}{\pi (1-s)}\int_{0}^{\infty }\psi ^{\prime
}(z+y-1/2)y^{1-s}dy=I_{1}(s,z)+I_{2}(s,z).  \notag
\end{gather}%

First, we deal with $I_{1}(s,z)$. By \eqref{eqSelZetaBound}, $\frac{Z^{\prime }(z+y)}{Z(z+y)}=O(y^{-n})$, for any positive integer $n$, as $y\rightarrow
\infty $. Also, $\frac{Z^{\prime }(z+y)}{Z(z+y)}%
=O(1)$, for fixed $z\in X_+$, as $%
y\rightarrow 0.$ Hence we may apply integration by parts to the integral $%
I_{1}(s,z)$ and obtain, for $0<\Re(s)<1$ and $z\in X_+,$
\begin{equation*}
I_{1}(s,z)=-\frac{\sin \pi s}{\pi (1-s)}\int_{0}^{\infty }y^{1-s}d\left(
\frac{Z^{\prime }(z+y)}{Z(z+y)}\right) =\frac{\sin \pi s}{\pi }%
\int_{0}^{\infty }\frac{Z^{\prime }(z+y)}{Z(z+y)}y^{-s}dy\text{.}
\end{equation*}%

The integral $I_1(s,z)$, for $z\in X_+$ is actually a holomorphic function in the half plane $\Re(s)<1$. To see this, let $\mu \leq 0$ be arbitrary. Since $(\log Z(z+y))' = O(N(P_0)^{-\Re(z+y)/2})$, as $y \to +\infty$, we have that $(\log Z(z+y))' = O(y^{-2+\mu})$, as $y\to +\infty$, where the implied constant may depend upon $z$ and $\mu$. Hence, $(\log Z(z+y))' y^{-s} = O(y^{-2})$,  as $y\to +\infty$, for all $s$ such that $\mu <\Re(s)\leq 0$. Moreover, the bound $\frac{Z^{\prime }(z+y)}{Z(z+y)}%
=O(1)$, for fixed $z\in X_+$ implies that $(\log Z(z+y))' y^{-s} = O(1)$,  as $y\to 0$, for all $s$ in the half plane $\Re(s)\leq 0$. This shows that for $z\in X_+$ the integral $I_1(s,z)$ is absolutely convergent in the strip $\mu <\Re(s)\leq 0$, hence represents a holomorphic function for all $s$ in that strip. Since $\mu \leq 0$ was arbitrarily chosen, we have proved that $I_1(s,z)$, for $z\in X_+,$ is holomorphic function in the half plane $\Re(s)\leq 0$.

Next, we claim that $I_1(s,z)$, for $z\in X_+,$ can be continued to the half-plane $\Re(s)>0$ as an entire function. For $z\in X_+$  and $0<\Re(s)<1$ we put
$$
\mathcal{I}_{1}(s,z)=\int_{0}^{\infty }\left( \frac{Z^{\prime }}{Z}%
(z+y)\right) y^{-s}dy
$$
and show that  for $z\in X_+$ the integral $\mathcal{I}_{1}(s,z)$ can be meromorphically continued to the half-plane $\Re(s)>0$
with simple poles at the points $s=1,2,...$ and corresponding residues
\begin{equation}
\mathrm{Res}_{s=n}\mathcal{I}_{1}(s,z)=-\frac{1}{(n-1)!}(\log
Z(z))^{(n)}.
\end{equation}
Since the function $\sin(\pi s)$ has simple zeros at points $s=1,2,...$ this would prove that $I_1(s,z)$, for $z\in X_+$ is actually an entire function of $s$.

Let $\mu>0$ be arbitrary, put $n=\lfloor \mu\rfloor$ to be the integer part of $\mu$ and let $\delta>0$ (depending upon $z\in X_+$ and $\mu$) be such that for $y\in (0,\delta)$ we have the Taylor series expansion
$$
(\log Z(z+y))'=\sum_{j=1}^{n}\frac{(\log Z(z))^{(j)}}{(j-1)!}y^{j-1} + R_1(z,y),
$$
where $R_1(z,y)=O(y^n)$, as $y\to 0$. Then, for $0<\Re(s)<1$ we may write
$$
\mathcal{I}_{1}(s,z)= \sum_{j=1}^{n}\frac{(\log Z(z))^{(j)}}{(j-1)!}\frac{\delta^{j-s}}{j-s} + \int_{0}^{\delta} R_1(z,y) y^{-s}dy + \int_{\delta}^{\infty }\left( \frac{Z^{\prime }}{Z}
(z+y)\right) y^{-s}dy.
$$
The bound on $R_1(z,y)$ and the bound \eqref{eqSelZetaBound} imply that the last two integrals are holomorphic functions of $s$ for $\Re(s)\in(0,\mu)$. The first sum is meromorphic in $s$, for $\Re(s)\in(0,\mu)$, with simple poles at $s=j$, $j\in\{1,...,n\}$ and residues equal to $-(\log
Z(z))^{(j)}/(j-1)!$. Since $\mu>0$ is arbitrary, this proves the claim. Therefore, we have proved that $I_1(s,z)$ is holomorphic function in the whole complex $s-$plane.

In order to evaluate integral $I_{2}(s,z)$ we use the fact that
$\psi^{\prime }(w)=\zeta_{H}(2,w)$ and that, for \mbox{$1<\Re(s)<2$}
\begin{equation*}
\overset{\infty }{\underset{k=0}{\sum }}\int_{0}^{\infty }\left\vert \frac{%
y^{1-s}dy}{\left( z+k-1/2+y\right) ^{2}}\right\vert \ll \overset{\infty }{%
\underset{k=0}{\sum }}\frac{1}{\left\vert z+k-1/2\right\vert ^{\text{Re}(s)}}%
<\infty.
\end{equation*}

For $z-p \notin \RR^-,$ and $0 < \Re(s) < 2,$ applying \cite[Formula 3.194.3.]{GR07}
we get
\begin{eqnarray*} \label{pIntegral}
\int_{0}^{\infty }\frac{y^{1-s}dy}{\left( z+y-p\right) ^{2}}
&=&\int_{0}^{\infty }\frac{y^{s-1}dy}{\left( 1+y\left[ z-p\right]
\right) ^{2}}=\frac{1}{\left( z-p\right) ^{s}}\cdot \frac{\Gamma
(s)\Gamma (2-s)}{\Gamma (2)}= \\
&=&\frac{(1-s)\Gamma (s)\Gamma (1-s)}{\left( z-p\right) ^{s}}=\frac{%
\pi (1-s)}{\sin \pi s}\cdot \frac{1}{\left( z-p\right) ^{s}},\notag
\end{eqnarray*}%
hence the dominated convergence theorem yields
\begin{equation*}
I_{2}(s,z)=\csp\frac{\sin \pi s}{\pi (1-s)}\overset{\infty }{\underset{k=0}{%
\sum }}\int_{0}^{\infty }\frac{y^{1-s}dy}{\left( z+k-1/2+y\right) ^{2}}=\csp
\overset{\infty }{\underset{k=0}{\sum }}\frac{1}{\left( z+k-1/2\right) ^{s}}%
=\csp\zeta_{H}(s,z-1/2).
\end{equation*}%

This, together with the representation \eqref{eqG1Hur} of $\G_1(s,z)$ and formula \eqref{zeta B plus G_1} proves \eqref{Z1rep} for $z\in X_+$. Moreover, for $z\in X_+$, the function $I_{2}(s,z)$ is meromorphic in the whole $s-$plane, with
a single simple pole at $s=1,$ with residue $\csp,$ hence the function $I_{1}(s,z)+I_{2}(s,z)$
is also meromorphic in the whole $s-$plane, with a single simple pole at $s=1,$ with residue $\csp$.

Combining this with Proposition \ref{prop:cont of G1} completes the proof.
\end{proof}

The polar structure of the superzeta function $\zeta_B^-(s,z)$, in the $s-$plane, for $z\in X_-$ is determined in the following theorem.

\begin{thm} \label{repZ2Thm}  For $z\in X_-$ the superzeta
function $\zeta_B^-(s,z)$ can be represented as

\beq \label{Z2rep}
\zeta_B^-(s,z)=-\mathcal{G}_{1}(s,z) + \csp \zeta_{H}(s,z)+\frac{\sin
\pi s}{\pi }\int_{0}^{%
\infty }\left( \frac{\left( ZH\right) ^\prime }{ZH}%
(z+y)\right) y^{-s}dy.
\eeq
Moreover, the superzeta function $\zeta_B^-(s,z)$, for $z\in X_-$, is a meromorphic function in variable $s$, with two simple poles at $s=1$ and $s=2$ with corresponding residues $\frac{\vol(M)}{\pi}(z-1/2) + \csp$ and $-\frac{\vol(M)}{\pi}$.
\end{thm}

\begin{proof}
The proof is very similar to the proof of Theorem~\ref{repZ1Thm}. We start with
\beq
Z_-(s)G_1(s) = \pi^{\csp/2}\exp(c_1s+c_2)\Gamma(s)^{-\csp} (ZH)(s),
\eeq
where the left-hand side of the equation is entire function of order two. Proceeding analogously as above, for $2<\Re(s)<3$ we get
\begin{multline*}
\zeta_B^-(s,z)+\mathcal{G}_{1}(s,z)=\frac{2\sin \pi s}{\pi
(1-s)(2-s)}\int_{0}^{\infty }\left[ \zeta_B^-(3,z+y)+\mathcal{G}%
_{1}(3,z+y)\right] y^{2-s}dy \\ = \frac{\sin \pi s}{\pi
(1-s)(2-s)}\int_{0}^{\infty }y^{2-s}d\left( \left( \log T\left(
z+y\right) \right) ^{\prime \prime }\right),
\end{multline*}
where
\begin{equation*}
\left( \log T\left( z+y\right) \right) ^{\prime \prime }=-\csp\psi
^{\prime }(z+y)+\left( \frac{(ZH)^{\prime }(z+y)}{ZH(z+y)}\right)
^{\prime }.
\end{equation*}

Bounds \eqref{asmPhi} and \eqref{eqSelZetaBound} imply that, for an arbitrary $\mu>0$, positive integer $k$ and $z\in X_-$  we have
$$
\frac{d^k}{dy^k}(\log (ZH)(z+y))=O(y^{-\mu}),\text{   as   } y\to+\infty,
$$
where the implied constant depends upon $z$ and $k$. Moreover, from the series representation of $Z(s)$ and $H(s)$ it is evident that $(\log(ZH)(z+y))'=O(1)$, as $y \to 0$.

Therefore, repeating the steps of the proof presented above we deduce that \eqref{Z2rep} holds true and that the  superzeta function $\zeta_B^-(s,z)$, for $z\in X_-,$ possesses meromorphic continuation to the whole complex $s-$plane with simple poles at $s=1$ and $s=2$ with residues $\frac{\vol(M)}{\pi}(z-1/2) + \csp$ and $-\frac{\vol(M)}{\pi}$, respectively.

\end{proof}

\section{Regularized determinant of the Lax-Phillips operator $B$}

After identifying the polar structure of the zeta functions $\zeta_B^{\pm}$, we are in position to state and prove our main results.

First, we express the complete zeta function $Z_+(z)$ as a regularized determinant of the operator $zI-(\frac{1}{2}I+B)$, modulo the factor of the form $\exp(\alpha_1 z+\beta_1)$, where $\alpha_1= \vol(M) \log(2\pi)/\pi $ and $\beta_1 = \frac{\vol(M)}{4\pi}( 4 \zeta'(-1) - \log(2\pi)) + \frac{\csp}{2}\log(2\pi)$ and obtain an analogous expression for the complete zeta function $Z_-(z)$, see Theorem 6.2. below.

Moreover, we prove that the scattering determinant $\phi(z)$ is equal to the product of $\exp(c_1z+c_2 + \frac{\csp}{2}\log\pi)$ and the quotient of regularized determinants of operators  $zI-(\frac{1}{2}I-B)$ and $zI-(\frac{1}{2}I+B)$.

Then, we define the higher depth regularized determinant, i.e. the regularized determinant of depth $r\in\{1,2,...\}$ and show that the determinant of depth $r$ of the operator $zI-(\frac{1}{2}I+B)$ can be expressed as a product of the Selberg zeta function of order $r$ and the Milnor gamma functions of depth $r$, see Theorem 6.4. below.

Finally, we express $Z'(1)$ in terms of the (suitably normalized) regularized determinant of $\frac{1}{2}I-B$.

\subsection{Regularized product associated to $G_1$, $Z_+,$ and $Z_-$}

A simple application of Proposition~\ref{prop: Voros cont.} yields expressions for regularized products associated to $G_1$, $Z_+,$ and $Z_-$.
We start with $\G_1(s,z)$, which is regular at $s=0$, hence we have the following proposition.

\begin{prop}
For all $z \in \mathbb{C} \backslash \left( -\infty ,0\right] $, the zeta regularized product of $\mathcal{G}_{1}(s,z)$
\ is given by
\begin{equation*}
D_{G_1}(z) =\exp \left(- \frac{\vol(M)}{2\pi } \left[ 2z \log (2\pi)+(2 \zeta'(-1)  -  \log(\sqrt{2\pi}))        \right] \right) G_{1}(z).
\end{equation*}
\end{prop}
\begin{proof}
From \eqref{def G_1} we get
$$\log{G_1(z)} =  \frac{\vol(M)}{2\pi}\left( z\log(2\pi) + 2\log{G(z+1)} - \log{\Gamma(z)} \right),  $$
upon applying \eqref{asmBarnes}  \eqref{gammaExpan}, (and $\zeta'(0)= -\tfrac{1}{2}\log(2\pi)$), and after a straightforward computation we obtain
\begin{multline*}
\log{G_1(z)} = \frac{\vol(M)}{2\pi}\left[ z^2(\log{z}-\tfrac{3}{2})- z(\log{z}-1) +(2 \log(2\pi))z \right. \\  \left. + \tfrac{1}{3}\log{z} - (\tfrac{1}{2}\log(2\pi) - 2\zeta'(-1))  \right] + \sum_{j=1}^{m-1} \frac{c_j}{z^j} + h_m(z),
\end{multline*}
where $c_j$ and $h_{m}(z)$ can be explicitly determined  from \eqref{asmBarnes} and \eqref{gammaExpan}
as $\Re(z) \to \infty$ in the sector $ |\arg z| < \frac{\pi}{2} - \delta,$ where $\delta > 0.$
Applying Proposition~\ref{prop: Voros cont.} with
\begin{equation*}
\widetilde{a_2}= \tfrac{\vol(M)}{2\pi}, b_2=0,  \widetilde{a_1}= -\tfrac{\vol(M)}{2\pi}, b_1 =   \tfrac{\vol(M)}{\pi}\log (2\pi), \widetilde{a_0}= \tfrac{\vol(M)}{6\pi}, b_0 = \tfrac{\vol(M)}{2\pi}(2 \zeta'(-1)  -  \log(\sqrt{2\pi}))
\end{equation*}
we obtain
\beq \label{regG1}
\exp \left( \left.- \frac{d}{ds}\G_{1}(s,z)\right\vert
_{s=0}\right) = \exp \left( -\frac{\vol(M)}{2\pi } \left[2z \log (2\pi) + (2 \zeta'(-1)  -  \log(\sqrt{2\pi}))        \right] \right) G_{1}(z)
\eeq

\end{proof}

Recall that $\text{Spec}\left( \frac{1}{2}I+B\right) =N (Z_+)$  and  $\text{Spec}\left( \frac{1}{2}I-B\right) = N(Z_-),$
hence  $\text{Spec}\left( zI - (\frac{1}{2}I + B)\right) = \{z - y_k ~|~ y_k \in   N(Z_+) \}$ and $\text{Spec}\left( zI -( \frac{1}{2}I - B)\right) = \{z - y_k ~|~ y_k \in   N(Z_-) \}.$

Therefore, for $z\in X_{\pm}$ we define

$$ \det \left( zI - (\frac{1}{2}I + B)\right) = D_{Z_+}(z) =  \exp \left( \left.- \frac{d}{ds}\zeta_B^+(s,z)\right\vert
_{s=0}\right),  $$
respectively
$$ \det \left( zI - (\frac{1}{2}I - B)\right) = D_{Z_-}(z) =  \exp \left( \left.- \frac{d}{ds}\zeta_B^-(s,z)\right\vert
_{s=0}\right).  $$

Our main result is

\begin{thm} \label{thmRegZ1}
For $z \in X_{\pm},$  the regularized product of $Z_{\pm}(z)$ is given by
\beq \label{regZi} \det \left( zI - (\frac{1}{2}I \pm B)\right)  =  \exp \left( \left.- \frac{d}{ds}\zeta_B^{\pm}(s,z)\right\vert
_{s=0}\right)  =  \\
  \Upsilon_{\pm}(z)  Z_{\pm}(z),
\eeq
where
\begin{equation*} \label{eqUp}
\Upsilon_+(z) = \exp\left[\frac{\vol(M)}{2\pi }\left( 2z \log (2\pi)+ 2 \zeta'(-1) - \frac{1}{2}\log(2\pi) + \frac{\csp \pi}{\vol(M)}\log(2\pi) \right) \right],
\end{equation*}
and
\begin{equation*} \label{eqUp_2}
\Upsilon_-(z) = \exp\left[\left(\frac{\vol(M)}{\pi } \log (2\pi) - c_1 \right) z + \frac{\vol(M)}{2\pi }(2 \zeta'(-1)  -  \log(\sqrt{2\pi})) -c_2  +\frac{\csp}{2}\log2     \right].
\end{equation*}
Moreover, for $z\in X_+\cap X_-$
\begin{equation} \label{phi formula}
\phi(z)= (\pi)^{\tfrac{\csp}{2}}e^{c_1z+c_2}\frac{\det \left( zI - (\frac{1}{2}I - B)\right)}{\det \left( zI - (\frac{1}{2}I + B)\right)} .
\end{equation}
\end{thm}

\begin{proof}
As $z \to \infty ,$  in $\Re(z) > 0,$ upon applying \eqref{asmBarnes}  \eqref{gammaExpan} we get
\begin{multline} \label{eqAsmZ1}
\log{Z_+(z)} =  \log{Z(z)} - \log{G_1(z)} - \csp
\log{\Gamma(z-\tfrac{1}{2}}) = \log{Z(z)} -  \\ -\frac{\vol(M)}{2\pi}\left[ z^2(\log{z}-\tfrac{3}{2})- z(\log{z}-1) +2z \log(2\pi) + \tfrac{1}{3}\log{z} - (\tfrac{1}{2}\log(2\pi) - 2\zeta'(-1))  \right]  \\
- \csp\left( \frac{1}{2}\log(2\pi) + (z-1)\log{z} - z   \right)  +  \sum_{j=1}^{m-1} \frac{c_j}{z^j} + h_m(z),
\end{multline}
where the $c_j$ and $h_m(z)$ can be calculated explicitly (with the help of Legendre's duplication formula).
By \eqref{eqSelZetaBound}, $\log{Z(z)}$ and its derivatives are of rapid decay, so it can be grouped with the last terms on the right.

Applying Proposition~\ref{prop: Voros cont.} with
\begin{align*}
\widetilde{a_2}= \tfrac{-\vol(M)}{2\pi},\quad b_2= 0, \quad \widetilde{a_1}= \tfrac{\vol(M)}{2\pi}-\csp,\quad b_1 =   \tfrac{-\vol(M)}{\pi} \log (2\pi), \\ \widetilde{a_0}= \tfrac{-\vol(M)}{6\pi}+\csp,\quad b_0 = \tfrac{\vol(M)}{2\pi}(\log(\sqrt{2\pi}) - 2 \zeta'(-1)  ) - \frac{\csp}{2}\log(2\pi),
\end{align*}
gives us the first part of \eqref{regZi}.

Next, to study $\zeta_B^-,$ recall that $Z_- = \phi Z_+.$ By \eqref{eqPhiA}, \eqref{asmPhi} and expansion
$$\log{L(z)} = \frac{\csp}{2}\log{\pi} + c_1z + c_2 + \csp
\log{\Gamma(z-\tfrac{1}{2}}) - \csp
\log{\Gamma(z)},  $$
we have, as $z \to \infty ,$  in $\Re(z) > 0,$
\begin{multline} \label{eqAsmZ2}
\log{Z_-(z)} =  \log{Z(z)} - \log{G_1(z)} -\csp\log{\Gamma(z)} + \frac{\csp}{2}\log{\pi} + c_1z + c_2   + \log{H(z)} \\ = - \frac{\vol(M)}{2\pi}\left[ z^2(\log{z}-\tfrac{3}{2}) - z(\log{z}-1) +2z \log(2\pi) \right. \\  \left. + \tfrac{1}{3}\log{z} - (\tfrac{1}{2}\log(2\pi) - 2\zeta'(-1))  \right] - \csp\left( \frac{1}{2}\log(2\pi) + (z-\tfrac{1}{2})\log{z} - z   \right)   \\ + \frac{\csp}{2}\log{\pi} + c_1z + c_2  + \log{Z(z)} + \log{H(z)} + \sum_{j=1}^{m-1} \frac{c_j}{z^j} + h_m(z)
\end{multline}

Note that we can group $ \log{Z(z)} + \log{H(z)}$ with the rapidly decaying remainder terms in \eqref{eqAsmZ2}. Applying Proposition~\ref{prop: Voros cont.} with
\begin{align*}
\widetilde{a_2}= -\tfrac{\vol(M)}{2\pi}, \quad b_2=0, \quad \widetilde{a_1}= \tfrac{\vol(M)}{2\pi}-\csp,\quad b_1 =  - \tfrac{\vol(M)}{2\pi}\cdot  2 \log (2\pi)+ c_1,\\ \widetilde{a_0}= -\tfrac{\vol(M)}{6\pi}+\frac{\csp}{2},\quad b_0 =- \tfrac{\vol(M)}{2\pi}(2 \zeta'(-1)  -  \log(\sqrt{2\pi})) + c_2-\frac{\csp}{2}\log2
\end{align*}
gives us the  second part of \eqref{regZi}.

It is left to prove \eqref{phi formula}. It follows after a straightforward computation from the relation $\phi(z)=Z_-(z)/Z_+(z)$ combined with \eqref{regZi}.

\end{proof}


%

\begin{rem}\rm
Equation \eqref{phi formula} shows that the scattering determinant, modulo a certain multiplication factor, is equal to a regularized determinant of the operator $\left( B+\left( z-\frac{1}{2}
\right) I\right) \left( R_{\left( z-\frac{1}{2}\right) }(B)\right) $,
for $z\in X_{+}\cap X_{-}$,  where $R_{\lambda }(B)$ denotes the resolvent of the
operator $B$. This  result is reminiscent of \cite[Theorem 1]{For87}, once we recall that $B$ is the infinitesimal generator of the
one-parameter  family $\mathbf{Z}(t)$. Namely, the right hand side
represents the quotient of regularized determinants of operators with
infinitely many eigenvalues, while the left hand side is a determinant of a
matrix operator.
\end{rem}

\begin{rem}\rm

Since the Selberg zeta function $Z(s)$ possesses a non-trivial, simple zero at $s=1$, it is obvious that $z=1/2 \notin X_{+}$. However, inserting formally $z=1/2$ into equation \eqref{phi formula} and recalling the fact that $\exp(c_1/2 + c_2) =d(1)/g_1$, we conclude that equation \eqref{phi formula} suggests that $\phi(1/2) = (-1)^{\csp}\mathrm{sgn}(d(1))$, where $\mathrm{sgn}(a)$ denotes the sign of a real, nonzero number $a$.
\end{rem}

\subsection{Higher-depth Determinants}

By Theorems \ref{repZ1Thm} and \ref{repZ2Thm}  functions $\zeta_B^+(s,z)$ and $\zeta_B^-(s,z)$ are
holomorphic at $s=0,-1,-2,...$, hence it is possible to define higher-depth
determinants of the operators $\left(zI -(\frac{1}{2}I+B)\right) $ and $\left(zI
-(\frac{1}{2}I-B)\right) $, as in \cite{KurWakYam}.

The determinant of the depth $r$ (where $r=1,2,...$) is defined for $z\in X_{\pm}$ as
\begin{equation}
\DET_{r}\left( zI-(\frac{1}{2}I+B)\right) =\exp \left( \left.
-\frac{d}{ds}\zeta_B^+(s,z)\right\vert _{s=1-r}\right)  \label{DetB1r}
\end{equation}%
and
\begin{equation*}
\DET_{r}\left( zI-(\frac{1}{2}I-B)\right) =\exp \left( \left.
-\frac{d}{ds}\zeta_B^-(s,z)\right\vert _{s=1-r}\right),  \label{DetB2r}
\end{equation*}%
respectively. When $r=1,$ we obtain the classical (zeta) regularized determinant.

The higher depth determinants of the operator $\left(zI -(\frac{1}{2}I+B)\right) $  can be expressed in terms of the Selberg zeta function $Z^{(r)} (s)$ of order $r\geq 1$ and the  Milnor gamma function of depth $r$, which is defined as
\begin{equation*}
\Gamma _{r}(z):=\exp \left( \left. \frac{\partial }{\partial w}\zeta
_{H}(w,z)\right\vert _{w=1-r}\right).
\end{equation*}
We have the following theorem
\begin{thm}
For $z \in X_+ ,$
and $r\in \mathbb{N}$\ one has
\begin{equation*}
\DET_{r}\left(zI -(\frac{1}{2}I+B)\right)  =  \Gamma _{r}(z-\frac{1}{2})^{-\csp}
\left[ Z^{(r)}(z)\right]^{\left[ (-1)^{r-1}(r-1)!\right]} \left( \frac{\Gamma_{r+1}(z)}{\Gamma_{r}(z)^{\left( z-\frac{1}{2}\right) }}\right)^{\frac{%
\vol(M)}{\pi }}.
\end{equation*}
\end{thm}
\begin{proof}

From Theorem~\ref{repZ1Thm}, for $z \in X_+ $ one has
 \begin{eqnarray*}
\zeta_B^+(s,z) &=&-\frac{\vol(M)}{\pi }\left[ \zeta
_{H}(s-1,z)-(z-1/2)\zeta _{H}(s,z)\right] +\csp \zeta _{H}(s,z-1/2)+ \\
&&+\frac{\sin \pi s}{\pi }\int_{0}^{\infty }\left( \frac{Z^{\prime }}{%
Z}(z+y)\right) y^{-s}dy
\end{eqnarray*}%
The right hand side is holomorphic, at $s=0,-1,-2,....$

Differentiating the above equation with respect to the variable $s$,
inserting the value $s=1-r$, where $r\in \mathbb{N}$ and having in mind the definition of
the Milnor gamma function of depth $r$ we get%
\begin{eqnarray}
\left. \frac{d }{ds}\zeta_B^+(s,z)\right\vert _{s=1-r}
&=&-\frac{\vol(M)}{\pi }\left[ \log \Gamma _{r+1}(z)-(z-\frac{1}{2})\log
\Gamma _{r}(z)\right] +\csp \log \Gamma _{r}(z-\frac{1}{2})+
\label{DerZetaB1At1-r} \\
&&+(-1)^{r-1}\int_{0}^{\infty }\left( \frac{Z^{\prime }}{Z}%
(z+y)\right) y^{r-1}dy  \notag
\end{eqnarray}%
for $z \in X_+.$

Assume for the moment that $\Re(z)>1$. Then, for $y\geq 0$
\begin{equation*}
\frac{Z^{\prime }}{Z}(z+y)=\sum_{P\in H(\Gamma )}\frac{\Lambda (P)}{%
N(P)^{z+y}}\text{.}
\end{equation*}
The absolute and uniform  convergence of the above sum for $\Re(z)>1$ and $y\geq 0$ imply that, for $r\in \mathbb{N}$
\begin{eqnarray*}
\int_{0}^{\infty }\left( \frac{Z^{\prime }}{Z}(z+y)\right) y^{r-1}dy
&=&\sum_{P\in H(\Gamma )}\frac{\Lambda (P)}{N(P)^{z}}\int_{0}^{\infty
}y^{r-1}\exp (-y\log N(P))dy= \\
&=&(r-1)!\sum_{P\in H(\Gamma )}\frac{\Lambda (P)}{N(P)^{z}\left( \log
N(P)\right) ^{r}}.
\end{eqnarray*}
Equation \eqref{DerZetaB1At1-r}, together with the above relation yield the formula
\begin{eqnarray*}
-\left. \frac{d }{ds}\zeta_B^+(s,z)\right\vert
_{s=1-r} &=&\frac{\vol(M)}{\pi }\log \left( \frac{\Gamma _{r+1}(z)}{\Gamma
_{r}(z)^{\left( z-\frac{1}{2}\right) }}\right) -\csp \log \Gamma _{r}(z-%
\frac{1}{2})+ \\
&&+(-1)^{r-1}(r-1)!\cdot \log Z^{(r)}(z),
\end{eqnarray*}
for $\Re(z)>1$. The statement of theorem follows by \eqref{DetB1r} and uniqueness of meromorphic continuation.
\end{proof}

\subsection{An expression for $Z'(1)$ as a regularized determinant}

Recall that for $z\in X_+$ we have
$$ \det \left( zI - (\frac{1}{2}I+ B)\right) = D_{Z_+}(z) =  \exp \left( \left.- \frac{d}{ds}\zeta_B^+(s,z)\right\vert
_{s=0}\right).  $$

The above regularized product is not well defined at $z=1$, since $z=1$ corresponds to the constant eigenfunction ($\lambda = 0$) of $\Delta$ of multiplicity one, hence it does not belong to $X_+$. Therefore, for $\Re(s)>2$ we define
$$(\zeta_B^+)^{*}(s,z) = \zeta_B^+(s,z) - (z-1)^{-s}= \sum_{\eta\in N(Z_+) \setminus\{1\}} \frac{1}{(z-\eta)^s}.$$

Meromorphic continuation of $(\zeta_B^+)^{*}(s,z)$ for $z\in X_+\cup \{1\}$ to the whole $s-$plane is immediate consequence of Theorem \ref{repZ1Thm} which implies that $(\zeta_B^+)^{*}(s,z)$ is holomorphic at $s=0$. Moreover,
$$
- \frac{d}{ds}(\zeta_B^+)^{*}(s,z)=- \frac{d}{ds}\zeta_B^+(s,z)-\frac{\log(z-1)}{(z-1)^s},
$$
hence
\begin{equation} \label{ZEtaB*}
\DET^{*} \left( I - (\frac{1}{2}I + B)\right) = \lim_{z \rightarrow 1}  \exp \left( \left.- \frac{d}{ds}(\zeta_B^+)^{*}(s,z)\right\vert_{s=0}\right) = \lim_{z \rightarrow 1} \frac{1}{z-1}D_{Z_+}(z)
,
\end{equation}

We give a direct proof of the following:
\begin{thm}\label{z1_as_det}
\begin{equation*}
\DET^{*} \left( -B+\frac{1}{2}I\right) =
2^\frac{\csp}{2} \exp\left[\frac{\vol(M)}{2\pi }\left(2 \zeta'(-1) + \frac{3}{2}\log(2\pi)  \right) \right] Z'(1).
\end{equation*}
\end{thm}

\begin{proof}
From \eqref{ZEtaB*} and Theorem~\ref{thmRegZ1} it follows that
\begin{multline*}
\DET^{*} \left(\frac{1}{2}I - B\right) =  \exp\left[\frac{\vol(M)}{2\pi }\left( 2 \log (2\pi)+ 2 \zeta'(-1) - \frac{1}{2}\log(2\pi) + \frac{\csp \pi}{\vol(M)}\log(2\pi)\right) \right] \cdot\lim_{z \rightarrow 1} \frac{Z_+(z)}{z-1}  \\ = \exp\left[\frac{\vol(M)}{2\pi }\left( \frac{3}{2}\log (2\pi)+ 2 \zeta'(-1) + \frac{\csp \pi}{\vol(M)}\log(2\pi) \right) \right] \frac{Z'(1)}{G_1(1)(\Gamma(1/2))^{\csp}} =\\
=2^\frac{\csp}{2} \exp\left[\frac{\vol(M)}{2\pi }\left(2 \zeta'(-1) + \frac{3}{2}\log(2\pi)  \right) \right] Z'(1).
\end{multline*}
\end{proof}

\begin{rem}\rm
The above corollary may be viewed as a generalization of
the result of \cite[Corollary 1]{Sarnak}, for the determinant $D_{0}$ to the
case of the non-compact, finite volume Riemann surface with cusps. Here, the
role of the Laplacian is played by the operator $-B+\frac{1}{2}I$. In
the case when $\csp=0$, the spectrum of $-B+\frac{1}{2}I$ consists of
points $s=\frac{1}{2}+ir_{n}$, $s=\frac{1}{2}-ir_{n}$, $r_{n}\neq 0,$
with multiplicities $m(\lambda _{n})$ and the point $s=\frac{1}{2}$ with
multiplicity $2d_{1/4}$. Therefore, formally speaking

\begin{equation*}
\underset{\lambda _{n}\neq 0}{\prod }\left( \frac{1}{2}+ir_{n}\right)
\left( \frac{1}{2}-ir_{n}\right) =\underset{\lambda _{n}\neq 0}{\prod }%
\lambda _{n}=\det \text{ }^{\prime }(\Delta _{0}),
\end{equation*}%
in the notation of Sarnak. For $\csp=0$, Corollary 3 may be viewed as the
statement%
\begin{equation*}
\det \text{ }^{\prime }(\Delta _{0})= \exp\left[\frac{\vol(M)}{2\pi }\left(2 \zeta'(-1) + \frac{3}{2}\log(2\pi)  \right) \right] Z'(1).
\end{equation*}%
This agrees with \cite{Sarnak}, the only difference being a constant term $%
\exp (-\frac{\vol(M)}{8\pi })$. It appears due to a different scaling
parameter we use. Namely, in \cite[Theorem 1]{Sarnak}, parameter is a
(natural for the trace formula setting) parameter $s(s-1)$, while we use $ s $ instead.
This yields to a slightly different asymptotic expansion at infinity and
produces a slightly different renormalization constant.
\end{rem}

\vspace{5mm}

\noindent
Joshua S. Friedman \\
Department of Mathematics and Science \\
\textsc{United States Merchant Marine Academy} \\
300 Steamboat Road \\
Kings Point, NY 11024 \\
U.S.A. \\
e-mail: FriedmanJ@usmma.edu, joshua@math.sunysb.edu, CrownEagle@gmail.com

\vspace{5mm}
\noindent
Jay Jorgenson \\
Department of Mathematics \\
The City College of New York \\
Convent Avenue at 138th Street \\
New York, NY 10031
U.S.A. \\
e-mail: jjorgenson@mindspring.com

\vspace{5mm}

\noindent
Lejla Smajlovi\'c \\
Department of Mathematics \\
University of Sarajevo\\
Zmaja od Bosne 35, 71 000 Sarajevo\\
Bosnia and Herzegovina\\
e-mail: lejlas@pmf.unsa.ba
\end{document}